\documentclass[12pt, reqno, a4paper, oneside]{amsart}

\usepackage{amsmath,amssymb,amsthm,eucal}
\usepackage[a4paper]{geometry}
\usepackage[dvipdfm,colorlinks=true]{hyperref}
\usepackage[all]{xy}
\usepackage{mathtools,here}

\newtheorem{thm}{Theorem}[section]

\newtheorem{lem}[thm]{Lemma}

\theoremstyle{definition}

\theoremstyle{remark}

\newcommand{\Z}{\mathbb{Z}}

\newcommand{\F}{\mathbb{F}}

\DeclareMathOperator{\Hom}{Hom}

\begin{document}

  \title{Cohomology of the spaces of commuting elements in Lie groups of rank two}
   
    \author{Masahiro Takeda}
    \address{Department of Mathematics, Kyoto University, Kyoto, 606-8502, Japan}
    \email{m.takeda@math.kyoto-u.ac.jp}


    \begin{abstract}
      Let $G$ be the classical group, and let $\Hom(\Z^m,G)$ denote the space of commuting $m$-tuples in $G$. Baird proved that the cohomology of $\Hom(\Z^m,G)$ is identified with a certain ring of invariants of the Weyl group of $G$. In this paper by using the result of Baird we give the cohomology ring of $\Hom(\Z^2,G)$ for simple Lie group $G$ of rank 2.
    \end{abstract}

  \maketitle

\setcounter{tocdepth}{1}


\section{Introduction}\label{Intro}

Let $G$ be a Lie group and $T$ be a maximal torus of $G$. Let $W(G)$ denote the Weyl group of $G$. The space of commuting elements in $G$, denoted by $\Hom(\Z^m,G)$, is the subspace of the Cartesian product $G^m$ consisting of $(g_1,\ldots,g_m)\in G^m$ such that $g_1,\ldots,g_m$ are pairwise commutative. Since the space $\Hom(\Z^m,G)$ is identified with the moduli space of based flat $G$-bundles over an $m$-torus, $\Hom(\Z^m,G)$ is studied in not only topology but also geometry and physics. On the other hand the cohomology of $\Hom(\Z^m,G)$ is deeply related with the invariant theory, since Baird proved that the cohomology is identified with a certain ring of invariants of the Weyl group of $G$. Thus the cohomology of $\Hom(\Z^m,G)$ is important for many fields. 
The general result on the cohomology of $\Hom(\Z^m,G)$ is studied in \cite{B,CS,RS1,RS2,KT}, while there is not much research on specific computation. For example, the space $\Hom(\Z^m,SU(2))$ is deeply studied in \cite{BJS,C}.

In this paper we give the cohomology ring of $\Hom(\Z^2,G)$ for $G=Sp(2),SU(3),G_2$. 
Let $\F$ be a field of characteristic zero or prime to the order of $W(G)$, and $\F\langle S \rangle$ denote a free graded commutative algebra generated by a graded set $S$.
The main theorem in this paper is the following.
\begin{thm}

For the simply connected simple Lie group $G$ of rank 2, there is an isomorphism
\[
H^*(\Hom(\Z^n,G);\F) \cong \F\langle a_1^1,a_2^1,a_1^2,a_2^2,b_1,b_2\rangle/(a_1^1,a_2^1,a_1^2,a_2^2,b_1,b_2)^3+I,
\]
where $I$ is generated by $$b_1b_2,\quad {b_2}^2,\quad a_2^1b_2,\quad a_2^2b_2,\quad a_1^1b_2+a_2^1b_1,\quad a_1^2b_2+a_2^2b_1,\quad  a_1^1a_2^2+a_1^2a_2^1,$$ and 
\[
|a_i^j|=
\begin{cases}
2i+1& (G=SU(3))\\
4i-1& (G=Sp(2))\\
8i-5& (G=G_2),
\end{cases}
\]
\[
|b_i|=
\begin{cases}
2i& (G=SU(3))\\
4i-2& (G=Sp(2))\\
8i-6& (G=G_2).
\end{cases}
\]

\end{thm}


To prove this we use the general results of Baird \cite{B}. Let $\Hom(\Z^m,G)_1$ denote the connected component of $\Hom(\Z^m,G)$ containing $(1,\dots ,1)$. 
We consider the action of $W(G)$ on $G/T\times T^m$ given by 
\[
w\cdot (gT,t_1,t_2, \dots , t_m)=(gwT,w^{-1}t_1w, \dots ,w^{-1}t_1w)
\]
for $w\in W(G),$ $g \in G$, $t_1,\dots ,t_m\in T$. Then the map 
\[
G\times T^m\to\Hom(\Z^m,G)_1,\quad(g,t_1,\ldots,t_m)\mapsto(gt_1g^{-1},\ldots,gt_mg^{-1})
\]
for $g \in G$, $t_1,\dots ,t_m\in T$ defines a map
\[
\phi \colon G/T\times_{W(G)}T^m\rightarrow \Hom(\Z^m,G)_1.
\]
In \cite{B} Baird proved that the map $\phi$ is an isomorphism in cohomology with $\F$ coefficient. Moreover by Theorem 4.1 in \cite{B} (proved by Kac and Smilga \cite{KS}) the space $\Hom(\Z^2,G)$ is connected for a 1-connected Lie group $G$. 
Thus for a 1-connected Lie group $G$ there is a ring isomorphism
\begin{equation}\label{Baird}
  H^*(\Hom(\Z^2,G);\F)\cong (H^*(G/T;\F)\otimes H^*(T;\F)^{\otimes 2})^{W(G)}.
\end{equation}

Using the result of Baird, we can see another meaning of the main theorem.
By Solomon's Theorem (see \cite[Theorem 9.3.2]{S}) we know the isomorphism for a simple Lie group $G$
\[
 (H^*(G/T;\F)\otimes H^*(T;\F))^{W(G)} \cong \Lambda(x_1,x_2,\dots x_n),
\]
where $n$ is the rank of the Lie group $G$ and $|x_i|$ is coincide with the $i$-th degree of $W(G)$ minus 1 \cite[Table1 in P.59]{H}. 
By this theorem the ungraded ring structure of $(H^*(G/T;\F)\otimes H^*(T;\F))^{W(G)}$ only depends on the rank of $G$.
We have not known the generalization of the Solomon's Theorem for the ring $(H^*(G/T;\F)\otimes H^*(T;\F)^{\otimes n})^{W(G)}$, but the main theorem in this paper supports the existence of the generalization.

We prove the main theorem for each case that $G=SU(3),Sp(2),G_2$.


\section{Cohomology of $\Hom(\Z^2,SU(2))$}

In this section we compute the ring structure of $H^*(\Hom(\Z^2,SU(3)))$. 
The result of this ring structure is recorded in \cite{KT}. 
By \eqref{Baird}, there is an isomorphism
$$H^*(\Hom(\Z^2,U(3))_1;\F)\cong \left(\Z[x_1, x_2,x_3]_{W(U(3))} \otimes \Lambda(y_1^1,y_2^1,y_3^1) \otimes \Lambda(y_1^2,y_2^2,y_3^2)\right)^{W(U(3))}, $$ 
where $|x_i|=2$ and $|y_i^j|=1$ for any $i,j$, $\Z[x_1, x_2,x_3]_{W(U(3))}$ is the ring of coinvariant of $W(U(3))$, and the actions of $W(U(3))$ on $\{ x_1,x_2,x_3\}$ and $\{ y_1^i,y_2^i,y_3^i \}$ are the permutation actions. By the isomorphism in Kishimoto and Takeda \cite[p.12 (4)]{KT}, there is an isomorphism
$$H^*(\Hom(\Z^2,SU(3)))\cong H^*(\Hom(\Z^2,U(3)))/(y_1^1+y_2^1+y_3^1,y_1^2+y_2^2+y_3^2).$$
By Theorem 1.1 in Ramras and Stafa \cite{RS1}, we obtain the Poincar\'e series of $\Hom(\Z^2,SU(3))$.
\begin{lem}\label{PoincareSU}

The Poincar\'e series of $\Hom(\Z^2,SU(3))$ is given by
\[
P(\Hom(\Z^2,SU(3))_1;t)=1+{t^{2}}+2 {t^{3}}+2 {t^{4}}+4 {t^{5}}+{t^{6}}+2 {t^{7}}+3 {t^{8}}.
\]

\end{lem}
In this section we define $a_i^{j}=x_1^{i}y_1^j+x_2^{i}y_2^j+x_3^{i}y_3^j$ and $b_i =x_1^{i-1}y_1^1y_1^2+x_2^{i-1}y_2^1y_2^2+x_3^{i-1}y_3^1y_3^2$.
\begin{thm}

  \label{minimal generator SU}
  $H^*(\Hom(\Z^2,SU(3)))$ is minimally generated by $\{a_1^1,a_2^1,a_1^2,a_2^2,b_1,b_2\}$.
  
\end{thm}

\begin{proof}

This follows from the result of Kishimoto and Takeda \cite[Corollary 6.18]{KT}.

\end{proof}

\begin{thm}\label{ringSU}

There is an isomorphism
\[
H^*(\Hom(\Z^2,SU(3));\F) \cong \F\langle a_1^1,a_2^1,a_1^2,a_2^2,b_1,b_2\rangle/(a_1^1,a_2^1,a_1^2,a_2^2,b_1,b_2)^3+I,
\]
where $I$ is generated by $$b_1b_2,\quad {b_2}^2,\quad a_2^1b_2,\quad a_2^2b_2,\quad a_1^1b_2+a_2^1b_1,\quad a_1^2b_2+a_2^2b_1,\quad  a_1^1a_2^2+a_1^2a_2^1.$$

\end{thm}

\begin{proof}

By Lemma \ref{minimal generator SU}, it remains to show that the generators of $(a_1^1,a_2^1,a_1^2,a_2^2,b_1,b_2)^3+I$ are $0$ in $H^*(\Hom(\Z^2,Sp(2));\F)$. By Lemma \ref{PoincareSU} it is sufficient to show that
\[
b_1b_2=0,\quad a_1^1b_2+a_2^1b_1=0,\quad a_1^2b_2+a_2^2b_1=0,\quad  a_1^1a_2^2+a_1^2a_2^1=0.
\]
There is a following equation
\begin{align*}
b_1b_2&=(y_1^1y_1^2+y_2^1y_2^2+y_3^1y_3^2)(x_1y_1^1y_1^2+x_2y_2^1y_2^2+x_3y_3^1y_3^2)\\
&=(x_1+x_2)y_1^1y_1^2y_2^1y_2^2+(x_2+x_3)y_2^1y_2^2y_3^1y_3^2+(x_3+x_1)y_3^1y_3^2y_1^1y_1^2\\
&=2(x_1+x_2+x_3)y_1^1y_1^2y_2^1y_2^2\\
&=0,
\end{align*}
and we obtain the first equation.
By the similar calculation the other equations can be shown.

\end{proof}


\section{Cohomology of $\Hom(\Z^2,Sp(2))$}

In this section we compute the ring structure of $H^*(\Hom(\Z^2,Sp(2)))$. By \eqref{Baird} there is an isomorphism
$$H^*(\Hom(\Z^2,Sp(2)))\cong \left(\Z[x_1, x_2]_{W(Sp(2))} \otimes \Lambda(y_1^1,y_2^1) \otimes \Lambda(y_1^2,y_2^2)\right)^{W(Sp(2))}, $$ 
where $|x_i|=2$, $|y_i|=1$, $\Z[x_1, x_2]_{W(Sp(2))}$ is the ring of coinvariant of $W(Sp(2))$, and the actions of $W(Sp(2))$ on $\{ x_1,x_2\}$ and $\{ y_1^i,y_2^i \}$ are the signed permutation actions. 
In this section we define $a_i^j=x_1^{2i-1}y_1^j+x_2^{2i-1}y_2^j$ and $b_i =x_1^{2i-2}y_1^1y_1^2+x_2^{2i-2}y_2^1y_2^2$. 

\begin{lem}

  \label{minimal generator Sp}
  $H^*(\Hom(\Z^2,Sp(2)))$ is minimally generated by $\{a_1^1,a_2^1,a_1^2,a_2^2,b_1,b_2\}$.
  
\end{lem}

\begin{proof}

This follows from the result of Kishimoto and Takeda \cite[Theorem 6.28]{KT}.

\end{proof}

On the other hand, by Theorem 1.1 in Ramras and Stafa \cite{RS1} we can determine the Poincar\'e series of $\Hom(\Z^2,Sp(2))$. 

\begin{lem}\label{PoincareSp}

The  Poincar\'e series of $\Hom(\Z^2,Sp(2))$ is given by
\[
P(\Hom(\Z^2,Sp(2));t)=1+t^2+2t^3+t^4+2t^5+2t^6+2t^{7}+2t^{9}+3t^{10}.
\]

\end{lem}

\begin{thm}\label{ringSp}

There is an isomorphism
\[
H^*(\Hom(\Z^n,Sp(2));\F) \cong \F\langle a_1^1,a_2^1,a_1^2,a_2^2,b_1,b_2\rangle/(a_1^1,a_2^1,a_1^2,a_2^2,b_1,b_2)^3+I,
\]
where $I$ is generated by $$b_1b_2,\quad {b_2}^2,\quad a_2^1b_2,\quad a_2^2b_2,\quad a_1^1b_2+a_2^1b_1,\quad a_1^2b_2+a_2^2b_1,\quad  a_1^1a_2^2+a_1^2a_2^1.$$

\end{thm}

\begin{proof}

By Lemma \ref{minimal generator Sp}, it remains to show that the generators of $(a_1^1,a_2^1,a_1^2,a_2^2,b_1,b_2)^3+I$ are $0$ in $H^*(\Hom(\Z^n,Sp(2));\F)$. By Lemma \ref{PoincareSp} it is sufficient to show that
\[
a_1^1b_2+a_2^1b_1=0,\quad a_1^2b_2+a_2^2b_1=0,\quad  a_1^1a_2^2+a_1^2a_2^1=0.
\]
There is a following equation
\begin{align*}
a_2^1b_1&=(x_1y_1^1+x_2y_2^1)(x_1^2y_1^1y_1^2+x_2^2y_2^1y_2^2)\\
&=x_1x_2^2y_1^1y_2^1y_2^2+x_1^2x_2y_2^1y_1^1y_1^2\\
&=-x_1^3y_1^1y_2^1y_2^2-x_2^3y_2^1y_1^1y_1^2\\
&=-(x_1^3y_1^1+x_2^3y_2^1)(y_1^1y_1^2+y_2^1y_2^2)\\
&=-a_1^1b_2,
\end{align*}
and we obtain the first equation.
By the similar calculation the second and third equations can be shown.

\end{proof}


\section{Cohomology of $\Hom(\Z^2,G_2)$}

In this section we compute the ring structure of $H^*(\Hom(\Z^2,G_2))$. The Weyl group $W(G_2)$ is isomorphic with the dihedral group $D_6=\langle a,b\mid a^6=b^2=abab=1\rangle $. Let $V=\langle z_1,z_2,z_3 \rangle$ be the three dimension vector space on $\F$ spanned by $z_1,z_2,z_3$ and a representation of $D_6$ on $V$ be the map $\phi\colon D_6 \rightarrow  GL(V)$ such that $\phi_a(s_1z_1+s_2z_2+s_3z_3)=-s_1z_3-s_2z_1-s_3z_2$ and $\phi_b(s_1z_1+s_2z_2+s_3z_3)=s_1z_1+s_2z_3+s_3z_2$ for $s_1,s_2,s_3\in \F$. Since $\phi_a(z_1+z_2+z_3)=-(z_1+z_2+z_3),\quad\phi_b(z_1+z_2+z_3)=z_1+z_2+z_3$, the representation $\phi$ has an invariant subspace $\langle z_1+z_2+z_3 \rangle$. Therefore there is a representation on $\langle z_1,z_2,z_3 \mid z_1+z_2+z_3=0 \rangle$ induced by $\phi$, and let $\bar{\phi}$ denote this representation. By the definition of $\bar{\phi}$ and the definition of the reflection representation of $W(G_2)$ (see \cite[Section 7]{AA}), we obtain the next lemma.

\begin{lem}\label{repG}

The action of $W(G_2)$ on $H^1(T;\F)$ and $H^2(G_2/T;\F)$ is equivalent to $\bar{\phi}$.

\end{lem}

Next we compute the ring of coinvariant of $W(G_2)$.

\begin{lem}\label{coinv isom}

There is an isomorphism

\[
\F[x_1,x_2,x_3]/(e_1,e_2,e_3^2)\rightarrow H^*(G_2/T;\F),
\]
where $|x_i|=2$ and $e_i$ is the $i$-th elementary symmetric polynomial in $x_1,x_2,x_3$.

\end{lem}

\begin{proof}
 
The cohomology $H^*(G_2/T;\F)$ is isomorphic with the ring of coinvariant of $W(G_2)$.
By the definition of $\phi$, the polynomial $e_2,e_3^2$ is in the invariant ring $\F[x_1,x_2,x_3]^{D_6}$. By Lemma \ref{repG} there is a surjection

\[
\alpha \colon \F[x_1,x_2,x_3]/(e_1,e_2,e_3^2) \rightarrow H^*(G_2/T;\F).
\]

Since $e_1,e_2,e_3^2$ is a regular sequence, the Poincar\'e series of $\F[x_1,x_2,x_3]/(e_1,e_2,e_3^2)$ is given by
\begin{align*}
P(\F[x_1,x_2,x_3]/(e_1,e_2,e_3^2);t)&=\left(\frac{1}{1-t^2}\right)^3(1-t^2)(1-t^4)(1-t^{12})\\
&=(1+t^2)(1+t^2+t^4+t^6+t^8+t^{10}),
\end{align*}
and of $H^*(G_2/T)$ is
\begin{align*}
P(G_2/T;t)&=\left(\frac{1}{1-t^2}\right)^2(1-t^4)(1-t^{12})\\
&=(1+t^2)(1+t^2+t^4+t^6+t^8+t^{10}).
\end{align*}
Since these Poincar\'e series are finite type, the map $\alpha$ is isomorphism. 

\end{proof}

\begin{lem}\label{D6coinv}

The set $\{x_1^ix_2^j\mid 0\leq i \leq 5, 0\leq j\leq 1\}$ is a bases of $\F[x_1,x_2,x_3]/(e_1,e_2,e_3^2)$.

\end{lem}

\begin{proof}

In $\F[x_1,x_2,x_3]/(e_1,e_2,e_3^2)$, there are equations
\begin{align*}
&x_1+x_2+x_3=e_1=0\\
&x_1^2+x_1x_2+x_2^2=-e_2+x_1e_1+x_2e_1=0.
\end{align*}
Therefore $x_3$ and $x_2^2$ can be replaced to $-x_1-x_2$ and $-x_1^2-x_1x_2$ respectively.
Since $x_2^3=-x_1^2x_2-x_1x_2^2=x_1^3$, there is a equation
\begin{align*}
x_1^6&=x_1^3x_2^3=x_1^2x_2^2(x_1+x_2)^2-x_1^2x_2^2(x_1^1+x_1x_2+x_2^2)\\
&=x_1^2x_2^2x_3^2=e_3^2=0.
\end{align*}
By considering the Poincar\'e series of $\F[x_1,x_2,x_3]/(e_1,e_2,e_3^2)$ which is given by 
\[
P(\F[x_1,x_2,x_3]/(e_1,e_2,e_3^2);t)=1+2t^2+2t^4+2t^6+2t^8+2t^{10}+t^{12},
\]
We obtain this lemma.

\end{proof}

By \eqref{Baird} and Lemma \ref{coinv isom}, there is an isomorphism
\[
H^*(\Hom(\Z^2,G_2)) \cong \left(\F[x_1,x_2,x_3]/(e_1,e_2,e_3^2)\otimes \bigotimes_{i=1}^2\Lambda(y_1^i,y_2^i,y_3^i)/(y_1^i+y_2^i+y_3^i)\right)^{D_6}.
\]

On the other hand, by Theorem 1.1 in Ramras and Stafa \cite{RS1} we can compute the Poincar\'e series of $\Hom(\Z^2,G_2)$.

\begin{lem}\label{PoincareG}

The  Poincar\'e series of $\Hom(\Z^2,G_2)$ is given by
\[
P(\Hom(\Z^2,G_2)_1;t)=1+t^2+2t^3+t^4+2t^5+t^6+t^{10}+2t^{11}+2t^{13}+3t^{14}.
\]

\end{lem}

In this section we define $a_i^j = \sum_{l=1}^3x_l^{4i-3}y_l^j$ and $b_i = \sum_{l=1}^3x_l^{4i-4}y_l^1y_l^2$.

\begin{thm}\label{ringG}

There is an isomorphism
\[
H^*(\Hom(\Z^n,G_2);\F) \cong \F\langle a_1^1,a_2^1,a_1^2,a_2^2,b_1,b_2\rangle/(a_1^1,a_2^1,a_1^2,a_2^2,b_1,b_2)^3+I,
\]
where $I$ is generated by $$b_1b_2,\quad {b_2}^2,\quad a_2^1b_2,\quad a_2^2b_2,\quad a_1^1b_2+a_2^1b_1,\quad a_1^2b_2+a_2^2b_1,\quad  a_1^1a_2^2+a_1^2a_2^1.$$

\end{thm}
\begin{proof}
First we prove that $H^*(\Hom(\Z^n,G_2);\F)$ is generated by $a_1^1,a_2^1,a_1^2,a_2^2,b_1,b_2$.
By the calculation similar to the proof in Lemma \ref{D6coinv}, there is a equation 
\begin{align*}
a_1^1b_2&=(\sum_{l=1}^3x_ly_l^1)(\sum_{l=1}^3x_l^{4}y_l^1y_l^2)\\
&=((x_1-x_3)y_1^1+(x_2-x_3)y_2^1)((x_1^4-x_3^4)y_1^1y_1^2+(x_2^4-x_3^4)y_2^1y_2^2)\\
&=(2x_2+x_1)(x_1^4-(x_1+x_2)^4)y_1^1y_2^1y_2^2+(2x_1+x_2)(x_2^4-(x_1+x_2)^4)y_2^1y_1^1y_1^2\\
&=3x_1^4x_2y_1^1y_2^1y_2^2+3x_1x_2^4y_2^1y_1^1y_1^2
\end{align*}
By Lemma \ref{D6coinv} we obtain $a_1^1b_2\ne0$. By the similar calculation we can show that $ a_1^2\ne0,a_1^1a_1^2\ne 0, a_1^ib_1\ne0, a_1^ib_2\ne 0, a_1^ia_2^j\ne 0$ for $i,j =1,2$. By Lemma \ref{PoincareG} and considering the degree with respect to the exterior algebra, we obtain that $H^*(\Hom(\Z^n,G_2);\F)$ is generated by $a_1^1,a_2^1,a_1^2,a_2^2,b_1,b_2$.

It remains to show that the ideal $(a_1^1,a_2^1,a_1^2,a_2^2,b_1,b_2)^3+I$ is $0$ in $H^*(\Hom(\Z^n,G_2);\F)$. By the calculation similar to Theorem\ref{ringSU}, we obtain this.

\end{proof}

\section*{Acknowledgement}

The author is grateful to Daisuke Kishimoto for valuable advice.


\end{document}